\documentclass[pdflatex,sn-mathphys-num]{sn-jnl}

\usepackage{graphicx}%
\usepackage{multirow}%
\usepackage{amsmath,amssymb,amsfonts}%
\usepackage{amsthm}%
\usepackage{mathrsfs}%
\usepackage[title]{appendix}%
\usepackage{xcolor}%
\usepackage{textcomp}%
\usepackage{manyfoot}%
\usepackage{booktabs}%
\usepackage{algorithm}%
\usepackage{algorithmicx}%
\usepackage{algpseudocode}%
\usepackage{listings}%
\usepackage{amsmath}
\usepackage{amsfonts}
\usepackage{amsthm}
\usepackage{mathtools}
\usepackage{amssymb}
\usepackage[shortlabels]{enumitem}

\newtheorem{theorem}{Theorem}[section]
\newtheorem{lemma}[theorem]{Lemma}

\newtheorem{conjecture}[theorem]{Conjecture}
\newtheorem{proposition}[theorem]{Proposition}
\newtheorem{definition}[theorem]{Definition}
\newtheorem{example}[theorem]{Example}
\newtheorem{problem}[theorem]{Problem}

\renewcommand{\epsilon}{\varepsilon}

\DeclarePairedDelimiterX{\norm}[1]{\lVert}{\rVert}{#1}
\DeclarePairedDelimiterX{\abs}[1]{\lvert}{\rvert}{#1}
\DeclarePairedDelimiter\bra{\langle}{\rvert}
\DeclarePairedDelimiter\ket{\lvert}{\rangle}
\DeclarePairedDelimiterX\braket[2]{\langle}{\rangle}{#1\,\delimsize\vert\,\mathopen{}#2}

\DeclarePairedDelimiterX{\gen}[1]{\langle}{\rangle}{#1}

\DeclareMathOperator{\tr}{tr}

\DeclareMathOperator{\Val}{Val}
\newcommand{\sub}{\mathfrak{sub}}
\newcommand{\IRS}{\operatorname{IRS}}
\DeclareMathOperator{\Stab}{Stab}

\usepackage{graphicx}

\newcommand{\Conv}{\mathop{\scalebox{1.5}{\raisebox{-0.2ex}{$\ast$}}}}

\begin{document}

\title[IRSs, Soficity, and L\"uck's determinant conjecture]{Invariant Random Subgroups, Soficity, and L\"uck's determinant conjecture}


\author*[1]{\fnm{Aareyan} \sur{Manzoor}}\email{a2manzoo@uwaterloo.ca}

\affil*[1]{\orgdiv{Department of Math}, \orgname{University of Waterloo}, \orgaddress{\street{University Avenue}, \city{Waterloo}, \postcode{N2L 3G1}, \state{Ontario}, \country{Canada}}}


\abstract{We extend Lück's determinant conjecture from groups to invariant random subgroups (IRS) of free groups, a framework generalizing groups where a non-sofic object is known to exist. For every free group, we prove the existence of an IRS satisfying the determinant conjecture that is not co-hyperlinear, and hence not co-sofic. This provides evidence that satisfying the determinant conjecture might be a weaker property than soficity for groups, and consequently the conjecture possibly holds for all groups. We use techniques from non-local games and $\mathsf{MIP}^* = \mathsf{RE}$, showing more generally when the latter can be used to narrow down when a von Neumann algebra (or IRS) contains a non-Connes embeddable object.}




\maketitle

\section{Introduction}

Let $\Gamma$ be a discrete group, $L(\Gamma)$ its group von Neumann algebra, and $\tau$ the standard trace on $L(\Gamma)$.  
We say that $\Gamma$ satisfies \textbf{L\"uck's determinant conjecture} if
\[
(\tau \otimes \mathrm{tr}_n)\big(\ln_+(A^*A)\big) \geq 0
\quad \forall\, A \in M_n(\mathbb{Z}[\Gamma]),\ n \in \mathbb{N},
\]
where $\mathrm{tr}_n$ is the normalized trace on $M_n(\mathbb{C})$, and
\[
\ln_+(t) =
\begin{cases}
0 & t = 0, \\
\ln(t) & t > 0.
\end{cases}
\]
When exponentiated, this quantity is called the \textbf{modified Fuglede--Kadison determinant} of $A$. In other words, $\ln_+(A^*A)$ is the logarithm of $A^*A$ with zeros ignored; exponentiating recovers the product of its non-zero singular values.

\begin{conjecture}[\cite{luck_l2-invariants_2002}]
Every discrete group satisfies L\"uck's determinant conjecture.
\end{conjecture}

The determinant conjecture is deeply connected to both operator algebras and topology. 
For instance, if a finitely generated group satisfies it and has vanishing first $\ell^2$-Betti number, 
then $L(\Gamma)$ is strongly $1$-bounded \cite{hayes_vanishing_2023,shlyakhtenko_von_2021}, 
a rigidity property in free probability that rules out embeddings of free group factors. 
It is also essential for the definition of $L^2$-torsion \cite{leary_l2-invariants_2009}, 
a powerful invariant in geometry and topology. 

The determinant conjecture is known to hold for all \textbf{sofic groups} \cite{elek_hyperlinearity_2005}. Whether every group is sofic remains an open problem, though it is widely believed that non-sofic groups exist.  
This makes it difficult to distinguish the determinant conjecture from soficity solely in the setting of groups.

We can circumvent this problem by working in the framework of \textbf{invariant random subgroups} (IRSs) of the free group. IRSs provide a setting where non-sofic analogues are known to exist, allowing us to separate the determinant conjecture from soficity in a way not possible for groups themselves.

An IRS is a conjugation-invariant Borel probability measure on the set of subgroups of a group.  
Dirac measures on normal subgroups are examples of IRSs, so this concept generalizes the notion of a normal subgroup.  
Since normal subgroups of free groups correspond to all finitely generated groups via quotients, IRSs of a free group can be seen as a natural generalization of groups themselves.  
The IRS analogue of soficity is called \textbf{co-soficity}.

Recent work has shown that every free group admits a non co-sofic IRS \cite{bowen_aldous--lyons_2024,bowen_aldous--lyons_2024-1}.  
In \cite{manzoor_there_2025}, the methods were significantly simplified and extended to produce a \textbf{non-co-hyperlinear} IRS on each free group—one failing the natural IRS analogue of hyperlinearity.

We can extend the determinant conjecture from groups to IRSs in a straightforward way, leading to the following central question:

\begin{problem}\label{prob:1.2}
Do all IRSs of a free group satisfy L\"uck's determinant conjecture?
\end{problem}

While we cannot yet resolve this fully, our main result provides evidence towards a positive answer.

\begin{theorem}\label{thm:1.3}
Let $\Gamma$ be a free group. There exists an IRS $H$ on $\Gamma$ that satisfies L\"uck's determinant conjecture but is not co-hyperlinear.
\end{theorem}

Such an $H$ cannot be co-sofic, as co-sofic IRSs are co-hyperlinear.  
This separation suggests that the determinant conjecture may be strictly weaker than soficity, strengthening the case that it could hold for \textbf{all} groups, including potential non-sofic ones. Previously, the determinant conjecture was thought of as ``an integral cousin of the Connes embedding property'' \cite{lueck_l2-torsion_2010}, but this shows that is wrong.

Note that the measured determinant conjecture \cite{lueck_l2-torsion_2010} implies a positive solution to Problem~\ref{prob:1.2}, so our result also provides evidence for the measured determinant conjecture. Actually, our results almost imply that there is a non-sofic equivalence relation satisfying the determinant conjecture. The obstruction is moving from an IRS to the full relation algebra.

Our proof of Theorem~\ref{thm:1.3} adapts the approach of \cite{manzoor_there_2025}, which is based on \textbf{non-local games}.  
This framework, together with the computability result $\mathsf{MIP}^* = \mathsf{RE}$
 \cite{ji_mipre_2022}, was used to prove the existence of a non--Connes embeddable von Neumann algebra.  
Much of the recent progress in narrowing down possible counterexamples to Connes' embedding problem has relied on this method.  
A well-known criticism, however, is that the argument is highly non-constructive: $\mathsf{MIP}^* = \mathsf{RE}$
 ensures existence but does not produce a useful explicit example.

However, there is a lot of interesting work that is orthogonal to the problem of giving explicit descriptions of examples. If one had an explicit IRS that is not co-sofic, it will still be extremely hard to prove it satisfies the determinant conjecture. On the other hand, the undecidability result gives it to us easily. See also \cite{lyons_ineffectiveness_2025} for another application of Aldous-Lyons that needed the undecidability result; an explicit example wouldn't have sufficed.

Section~\ref{sec:3.1} serves as a proof of concept.  
Suppose there exists a ``nice'' class of algebras (resp.\ IRSs) that:
\begin{enumerate}
    \item admits nice upper bounds, and
    \item contains all finite-dimensional tracial von Neumann algebras (resp.\ IRSs).
\end{enumerate}
Then $\mathsf{MIP}^* = \mathsf{RE}$
 implies that this class must contain a non--Connes embeddable algebra (resp.\ IRS).  
Similarly, by $\mathsf{MIP}^{\mathsf{co}} = \mathsf{coRE}$
, if a ``nice'' class admits computable lower bounds, it cannot contain \textbf{every} tracial von Neumann algebra (resp.\ IRS). This means the class of all tracial von Neumann algebras (resp. IRS) has no ``nice'' approximation property that encompasses the whole class. The IRSs satisfying the determinant conjecture fall in the former category, which allows us to get the result. See Proposition \ref{prop:mip-nonemb} for the formal statement.

In Section \ref{sec:2} we extend the determinant conjecture to IRSs. Section \ref{sec:3} develops the non-local game framework and proves Theorem \ref{thm:1.3}, while Section \ref{sec:4} discusses broader implications.

Section~\ref{sec:3.1} develops a NPA hierarchy for strategies coming from IRSs satisfying the determinant conjecture, which forms the technical core of the paper. This hierarchy not only yields Theorem~\ref{thm:1.3}, but also provides a template for applying non-local game techniques to other approximation problems involving IRSs and traces/von Neumann algebras.

\subsection*{Acknowledgements}
    I would like to thank Michael Brannan for discussions on these results and supervision. I would also like to thank Georges Skandalis and Balci for sharing Balci's PhD thesis with me, which helped inspire the NPA hierarchy in this paper. I would also like to thank Mustafa Nawaz for proofreading this paper. Any errors 
    are my own.

\section{Invariant Random Subgroups}\label{sec:2}

In this section we introduce the basic terminology for \textbf{invariant random subgroups} (IRSs) and recall how they give rise to canonical traces and von Neumann algebras. 
For a more detailed discussion, see \cite[Section~2]{manzoor_there_2025}.

\begin{definition}
    Let $\Gamma$ be a discrete group.  
    Denote by $\sub(\Gamma)$ the space of subgroups of $\Gamma$, topologized as a subset of $\{0,1\}^\Gamma$.  
    A \textbf{random subgroup} of $\Gamma$ is a Borel probability measure on $\sub(\Gamma)$.  
    If the measure is invariant under the conjugation action of $\Gamma$ on $\sub(\Gamma)$, we call it an \textbf{invariant random subgroup} (IRS).  
    We write $\IRS(\Gamma)$ for the convex set of IRSs on $\Gamma$.
\end{definition}

\begin{example}
    If $N \trianglelefteq \Gamma$ is a normal subgroup, then the Dirac measure $\delta_N$ is an IRS.
\end{example}

Thus, IRSs generalize the notion of normal subgroups.  
In what follows we focus on IRSs of free groups.  
Since normal subgroups of a free group correspond to all finitely generated groups via quotients, IRSs of a free group can be viewed as a generalized notion of a group, obtained through a probabilistic ``quotient'' construction (see Definition~\ref{def:1.11}).

\begin{example}
    Let $\alpha \colon \Gamma \to S_n$ be an action of $\Gamma$ on a finite set of $n$ elements.  
    Define the IRS $H_\alpha$ by
    \[
        H_\alpha(A) = \frac{\#\{\,i \in [n] : \Stab(i) \in A \,\}}{n} \quad \text{for } A \subset \sub(\Gamma).
    \]
    In words: choose a random point in $[n]$, and take its stabilizer.  
    Such IRSs are called \textbf{finitely described}.
\end{example}

The space $\sub(\Gamma)$ is compact in the topology inherited from $\{0,1\}^\Gamma$, so weak$^*$ convergence of measures on $\sub(\Gamma)$ makes sense.

\begin{example}
    An IRS $H \in \IRS(\Gamma)$ is called \textbf{co-sofic} if it is the weak$^*$ limit of finitely described IRSs.
\end{example}

\begin{proposition}\cite[Lemma 16]{abert_rank_2017}\label{prop:1.14}
    Let $\Gamma$ be a free group and $N \trianglelefteq \Gamma$.  
    Then $\delta_N$ is co-sofic as an IRS if and only if $\Gamma/N$ is sofic as a group.
\end{proposition}

\medskip

To connect IRSs to operator algebras we pass through traces on group algebras.  
A \textbf{trace} on $\Gamma$ (also called a \textbf{character}) is a conjugation-invariant, positive-definite function $\tau \colon \Gamma \to \mathbb{C}$ with $\tau(e)=1$.  
Equivalently, traces correspond to tracial states on the universal group $C^*$-algebra $C^*(\Gamma)$.

    The \textbf{full group $C^*$-algebra} $C^*(\Gamma)$ is the unique unital $C^*$-algebra generated by unitaries $\{u_g : g \in \Gamma\}$ subject to the relations $u_g u_h = u_{gh}$, with the universal property that any unitary representation $\pi \colon \Gamma \to U(\mathcal{H})$ extends uniquely to a $*$-homomorphism $C^*(\Gamma)\to B(\mathcal{H})$.  
    A \textbf{tracial state} on $C^*(\Gamma)$ is a linear map $\tau:C^*(\Gamma)\to \mathbb{C}$ such that $\tau(a^*a)\geq 0$ and $\tau(ab)=\tau(ba)$ for all $a,b$.

Given a tracial state $\tau$ on $C^*(\Gamma)$, the \textbf{Gelfand–Naimark–Segal (GNS) construction} produces:
\begin{itemize}
    \item a Hilbert space $\mathcal{H}_\tau$,
    \item a unitary representation $\pi_\tau : \Gamma \to U(\mathcal{H}_\tau)$,
    \item and a cyclic vector $\xi_\tau \in \mathcal{H}_\tau$
\end{itemize}
such that 
\[
    \tau(g) = \bra{\xi_\tau}\pi_\tau(g)\ket{\xi_\tau}.
\]
The von Neumann algebra generated by the image $\pi_\tau(C^*(\Gamma))''\subset B(\mathcal{H}_\tau)$ then carries the unique extension of the trace $\tau$, and is called the \textbf{tracial von Neumann algebra associated to $\tau$}. 

\medskip

Now let $H \in \IRS(\Gamma)$.  
There is a canonical way to define a trace from $H$: 

\begin{definition}
    For $g \in \Gamma$, set
    \[
        \tau_H(g) := \mathbb{P}(g \in H).
    \]
\end{definition}

This $\tau_H$ extends to a tracial state on $C^*(\Gamma)$; invariance of $H$ ensures $\tau_H$ is a trace, and positivity follows from the GNS construction \cite[Proposition~2.11]{manzoor_there_2025}.  
The associated von Neumann algebra captures the idea of ``quotienting by an IRS.''

\begin{definition}\label{def:1.11}
    Let $\Gamma$ be a discrete group and $H \in \IRS(\Gamma)$.  
    The \textbf{quotient of $\Gamma$ by $H$} is the tracial von Neumann algebra
    \[
        L(\Gamma/H) := \pi_{\tau_H}(C^*(\Gamma))''.
    \]
\end{definition}

If $H = \delta_N$ for some normal subgroup $N \trianglelefteq \Gamma$, then $L(\Gamma/H) = L(\Gamma/N)$, the usual group von Neumann algebra of the quotient.

\medskip

Soficity for IRSs has an analogue in \textbf{hyperlinearity}, corresponding to Connes embeddability of $L(\Gamma/H)$. Roughly speaking soficity is approximation by permutation matrices, while hyperlinearity/Connes embeddebility is approximation by any matrices. See \cite[Section 2.2]{manzoor_there_2025} for detailed discussions.
The following characterizations clarify the trace perspective:

\begin{proposition}[\cite{brown_invariant_2006}]
    Let $\Gamma$ be a free group, and let $\tau$ be a trace on $\Gamma$.  
    The following are equivalent:
    \begin{enumerate}
        \item $\pi_\tau(C^*(\Gamma))''$ is Connes embeddable.
        \item $\tau$ is amenable as a trace.
        \item $\tau$ is the weak$^*$ limit of traces with finite-dimensional GNS representations.
    \end{enumerate}
\end{proposition}

\begin{definition}
    An IRS $H$ of a free group $\Gamma$ is \textbf{co-hyperlinear} if $L(\Gamma/H)$ is Connes embeddable; equivalently, if $\tau_H$ is amenable as a trace.
\end{definition}

\subsection{L\"uck's Determinant Conjecture for IRSs}

Let $\Gamma$ be a discrete group and $H \in \IRS(\Gamma)$.  
Define $\mathbb{Z}[\Gamma/H] \subset L(\Gamma/H)$ as the image of the $*$-subalgebra $\mathbb{Z}[\Gamma] \subset C^*(\Gamma)$ under the GNS representation associated to $\tau_H$.  
If $H = \delta_N$ for some normal subgroup $N \trianglelefteq \Gamma$, then $\mathbb{Z}[\Gamma/H] = \mathbb{Z}[\Gamma/N] \subset L(\Gamma/N)$.

\begin{definition}
    Let $\Gamma$ be a discrete group and $H$ an IRS on it.  
    We say that $H$ \textbf{satisfies the determinant conjecture} if
    \[
        (\tau_H \otimes \tr_n)\big(\ln_+(A^*A)\big) \geq 0
        \quad \forall\, A \in M_n(\mathbb{Z}[\Gamma/H]), \ n \in \mathbb{N},
    \]
    where $\tr_n$ is the normalized trace on $M_n(\mathbb{C})$ and
    \[
        \ln_+(t) =
        \begin{cases}
            0 & t = 0 \\
            \ln(t) & t > 0.
        \end{cases}
    \]
\end{definition}

This condition depends only on the trace $\tau_H$:

\begin{definition}[\cite{balci_traces_2016}]
    Let $\tau$ be a trace on $C^*(\Gamma)$ where $\Gamma$ is a discrete group.  
    We say that $\tau$ satisfies the determinant conjecture if
    \[
        (\tau \otimes \tr_n)\big(\ln_+(A^*A)\big) \geq 0
        \quad \forall\, A \in M_n(\mathbb{Z}[\Gamma]), \ n \in \mathbb{N}.
    \]
\end{definition}

\begin{proposition}
    Let $\Gamma$ be a discrete group and $H \in \IRS(\Gamma)$.
    \begin{enumerate}
        \item If $H = \delta_N$ for some normal subgroup $N$, then $H$ satisfies the determinant conjecture as an IRS if and only if $\Gamma/N$ satisfies it as a group.
        \item $H$ satisfies the determinant conjecture as an IRS if and only if $\tau_H$ satisfies it as a trace.
    \end{enumerate}
\end{proposition}

As sofic groups satisfy the determinant conjecture, it would be a reasonable guess that the property also holds for co-sofic IRSs in free groups.  Indeed, this essentially appears in \cite{balci_traces_2016} in different language:

\begin{lemma}\label{lem:2.13}
    Let $\Gamma$ be a discrete group and $H$ a co-sofic IRS on it.  
    Then $H$ satisfies the determinant conjecture.
\end{lemma}

\begin{proof}
    First consider a finitely described IRS.  
    Let $\alpha \colon \Gamma \to S_n$ induce the IRS $H$.  
    Then $\tau_H = \tr_n \circ \alpha$, where $S_n$ is regarded as a subset of $M_n(\mathbb{C})$.  
    If $A \in M_k(\mathbb{Z}[\Gamma])$, then $\alpha(A) \in M_{nk}(\mathbb{Z})$, and
    \[
        (\tau_H \otimes \tr_k)\big(\ln_+(A^*A)\big)
        = \tr_{nk}\big(\ln_+(\alpha(A)^*\alpha(A))\big).
    \]
    Exponentiating, this becomes the product of the non-zero eigenvalues of $\alpha(A)^*\alpha(A)$, i.e., the smallest non-zero coefficient of its characteristic polynomial.  
    Since $\alpha(A)^*\alpha(A)$ is positive and has integer entries, this product is at least $1$, hence $(\tau_H \otimes \tr_k)(\ln_+(A^*A)) \geq 0$.

    Now suppose $H_n \to H$ in the weak$^*$ topology, where each $H_n$ is finitely described.  
    The map $H \mapsto \tau_H$ is $w^*\!-\!w^*$ continuous from $\IRS(\Gamma)$ to the space of traces $T(\Gamma)$, so $\tau_{H_n} \to \tau_H$ pointwise.  
    The inequality then passes to the limit.
\end{proof}

This lemma is actually integral in our proof of Theorem \ref{thm:1.3}, as we need the permutation strategies to also be Det-IRS strategy (see Definition \ref{def:3.2} (v) and Proposition \ref{prop:mip-nonemb} (B)).

Let $H\in \IRS(\Gamma)$ for discrete group $\Gamma$. Note there is a probability measure preserving equivalence relation on standard Borel space $\mathcal{R}_H$ so that $L(\Gamma/H)\subset L(\mathcal{R}_H)$ with $\mathbb{Z}(\Gamma/H)\subset \mathbb{Z}(\mathcal{R}_H)$ \cite[proposition 2.11]{manzoor_there_2025}. In particular if $\mathcal{R}_H$ satisfies L\"uck's determinant conjecture in the sense of \cite{lueck_l2-torsion_2010}, then so does $H$. This situates the IRS version within the broader framework of measured group theory: it shows that evidence for the IRS determinant conjecture also points towards the validity of L\"uck’s conjecture for measured equivalence relations.

While an IRS satisfying the determinant conjecture need not force the associated equivalence relation to do so, the only obstruction arises from the intermediate $L^\infty(X)$ algebra. This motivates the following conjecture:
\begin{conjecture}\label{conj:2.14}
    Suppose $\Gamma \overset{\alpha}{\curvearrowright} (X,\mu)$ is a probability measure preserving action. If $H_\alpha$ satisfies the determinant conjecture as an IRS, so does the orbit equivalence relation $\mathcal{R}_\alpha.$
\end{conjecture}

A positive resolution of this conjecture would yield, for the first time, a non-sofic equivalence relation satisfying the determinant conjecture, from Theorem \ref{thm:1.3}. Moreover, it would show that the measured determinant conjecture for groups coincides with the usual determinant conjecture, by specializing to the case of free actions.

\section{Non-Local Games}\label{sec:3}

In this section, we recall the connection between traces on certain group $C^*$-algebras and strategies for non-local games.  
For background and a more detailed discussion, see \cite[Sections~4--5]{goldbring_connes_2021}.  
We follow the notational conventions of \cite{manzoor_there_2025}.

\begin{definition}
    A \textbf{non-local game} $\mathfrak{G}$ has the following parameters:
    \begin{itemize}
        \item a finite \textbf{question set} $Q$;
        \item a probability distribution $q$ on $Q \times Q$ (the questions asked jointly to Alice and Bob);
        \item an integer $m \geq 1$ describing the length of each player's answer in bits;
        \item a \textbf{decider function} 
        \[
            D : (\{0,1\}^m)^2 \times Q^2 \to \{0,1\}, 
        \]
        where $D(a,b |x,y)$ determines whether the verifier accepts the answers $a,b$ to questions $x,y$.
    \end{itemize}
\end{definition}

Usually such games are called synchronous in literature, but we will never consider non-synchronous games so we drop that adjective. One can think of a non-local game as an interactive proof system:  
a verifier samples $(x,y) \sim q$ from $Q \times Q$, sends $x$ to Alice and $y$ to Bob,  
who reply with $m$-bit strings $a,b \in \{0,1\}^m$.  
The verifier accepts if and only if $D(a,b |x,y) = 1$.  
A \textbf{strategy} for Alice and Bob is thus given by a collection of conditional probabilities  
\[
    \big(p(a,b|x,y) \big)_{a,b,x,y}.
\]

To analyze such games algebraically, we associate to $\mathfrak{G}$ a set of formal variables
\[
    S(\mathfrak{G}) = \{ u_{x,i} \;:\; x \in Q,\, 1 \leq i \leq m \},
\]
where $u_{x,i}$ represents the $i$th bit of the answer to question $x$.  
For brevity, we write $S$ when the game is fixed and $S_x$ for the variables corresponding to a fixed $x$.

Let $\mathcal{F}(S,2)$ be the group generated by the elements of $S$ subject to the relations $u^2 = 1$.  
Intuitively, we treat each answer bit as a unitary with eigenvalues $\pm 1$, with the eigenvalues corresponding to a possible output of measuring the quantum bit.
Next, we impose that answers to the same question commute:
\[
    [u_{x,i},u_{x,j}] = 1 \quad \text{for each } x \in Q,\ 1 \leq i,j \leq m.
\]
The resulting quotient group is denoted
\[
    \Conv_Q \mathbb{Z}_2^m.
\]
This group encodes the idea that each question $x$ has $m$ binary answers, with commuting relations among them.  
Its full group $C^*$-algebra $C^*(\Conv_Q \mathbb{Z}_2^m)$ serves as the universal algebra in which to model strategies.

Inside this algebra, the unitary corresponding to $u_{x,i}$ is denoted $U_{x,i}$.  
For each answer $a \in \mathbb{Z}_2^m$, define the spectral projection
\[
    e_x^a := \prod_{i \leq m} \frac{1 + (-1)^{a_i} U_{x,i}}{2}.
\]
These projections satisfy $e_x^a e_x^b = \delta_{ab} e_x^a$ and $\sum_a e_x^a = 1$,  
and can be interpreted as the ``quantum probability'' of answering $a$ to question $x$.

A trace on $C^*(\Conv_Q \mathbb{Z}_2^m)$ turns these ``quantum probabilities'' into actual probabilities we can work with, and hence strategies. By narrowing down the set of traces, we can get different kind of strategies on non-local games.

\begin{definition}\label{def:3.2}
    Let $\mathfrak{G}$ be a non-local game with question set $Q$ and answer length $m$,  
    and let $\mathcal{A} := C^*(\Conv_Q \mathbb{Z}_2^m)$.  
    We define:
    \begin{enumerate}
        \item A \textbf{synchronous quantum commuting strategy} if there exists a trace $\tau \in T(\mathcal{A})$ with  
        \[
           p(a,b|x,y) = \tau(e_x^a e_y^b).
        \]

        \item A \textbf{synchronous quantum approximate strategy} if the above trace $\tau$ is \textbf{amenable},  
        i.e.\ it is a weak$^*$-limit of finite-dimensional traces.

        \item An \textbf{IRS strategy} if there exists an IRS $H$ on  
        $\Conv_Q \mathbb{Z}_2^m \times \mathbb{Z}_2$ such that
        \[
           p(a,b|x,y) = \tau_H \big( (1-J) e_x^a e_y^b \big),
        \]
        where $J$ is the generator of $\mathbb{Z}_2$ in $C^*(\Conv_Q \mathbb{Z}_2^m \times \mathbb{Z}_2)$  
        (see \cite{manzoor_there_2025}).

        \item A \textbf{permutation strategy} if it is an IRS strategy coming from a finitely described IRS.

        \item A \textbf{Det-IRS strategy} if it is an IRS strategy coming from an IRS satisfying the determinant conjecture.
    \end{enumerate}
\end{definition}
We remark that definitions $(i)$ and $(ii)$ are not standard but are equivalent to the standard definition by \cite[Theorem 3.5]{kim_synchronous_2018}. The IRS strategies are, from a physics point of view, an extremely unnatural class of strategies. However, this artificial class is precisely what was needed to find a non co-hyperlinear IRS in \cite{manzoor_there_2025}. We use the Det-IRS strategies to replicate the methods of that paper in our setting.

The relevant inclusion relations can be summarized as
\[
\text{Permutation strategies } \subset \text{ Det-IRS strategies } \subset \text{IRS strategies}.
\]  

Let $\omega^*(\mathfrak{G})$ denote the optimal winning probability over quantum approximate strategies,  
and let $\omega_{\mathrm{det}}(\mathfrak{G})$ be the optimal winning probability over Det-IRS strategies.  
The main idea of \cite{manzoor_there_2025} was to find $\mathfrak{G}$ such that  
$\omega^*(\mathfrak{G})$ is strictly less than the best IRS winning probability,  
producing a non-co-hyperlinear IRS.  
We aim to do the same, but with the IRS additionally satisfying the determinant conjecture.  
In words: if an Det-IRS strategy can beat every quantum strategy in some game, then it comes from an IRS that satisfies the determinant conjecture but is not co-hyperlinear.

\begin{theorem}\label{thm:3.5}
    Suppose there exists a game $\mathfrak{G}$ with
    \[
        \omega_{\mathrm{det}}(\mathfrak{G}) > \omega^*(\mathfrak{G}).
    \]
    Then there exists a non-co-hyperlinear IRS on some free group which satisfies the determinant conjecture.
\end{theorem}

\begin{proof}
    The proof follows exactly as in \cite[Theorem~3.5]{manzoor_there_2025},  
    with the observation that if $H$ is an IRS on  
    $\Conv_Q \mathbb{Z}_2^m \times \mathbb{Z}_2$ satisfying the determinant conjecture,  
    then its lift to $\mathcal{F}(S \cup \{J\})$ also satisfies the determinant conjecture.
\end{proof}

\subsection{NPA hierarchy}\label{sec:3.1}

To separate $\omega_{\mathrm{det}}(\mathfrak{G})$ from $\omega^*(\mathfrak{G})$,  
we need computable upper bounds for $\omega_{\mathrm{det}}(\mathfrak{G})$,  
analogous to the \textbf{NPA hierarchy} \cite{navascues_convergent_2008} for quantum strategies:  
\begin{theorem}\label{thm:det-npa}
    For any non-local game $\mathfrak{G}$ there is a computably enumerable, monotonically decreasing sequence $\alpha_n \to \omega_{\mathrm{det}}(\mathfrak{G})$.
\end{theorem}
Here computably enumerable means there is a Turning machine that on input $\mathfrak{G}$ will enumerate the sequence $\{\alpha_n\}$.

We work in the setting of the previous subsection,  
fixing a non-local game $\mathfrak{G}$ with variable set $S$ and letting
\[
    \Gamma := \Conv_Q \mathbb{Z}_2^m \times \mathbb{Z}_2.
\]

This section will be technical, but the main point is we can take finite sets $B_n$ with $\bigcup_n B_n = \Gamma$. Then we can check if a strategy looks like it comes from an IRS satisfying the determinant conjecture locally on each $B_n$ (The IRS part was done in \cite[Section 3.2]{manzoor_there_2025}, the new contribution is from the determinant conjecture part). As $B_n$ gets bigger, the checks get stricter and we can use these to approximate the set of IRSs satisfying determinant conjecture from above.

For finite subsets $B \subset C \subset \Gamma$ we define the \textbf{restriction map}
\[
    R_{B \subset C} : \{0,1\}^C \to \{0,1\}^B,\quad
    A \mapsto A \cap B.
\]
When $C = \Gamma$, we write $R_B$. These restriction maps allow us to compare local data on different finite subsets.

The value of $\mathfrak{G}$ on a trace $\tau$ depends only on its restriction to a fixed finite subset $K \subset \Gamma$ \cite[Section~3.2]{manzoor_there_2025}.  
Thus for $B \supset K$ and $\pi \in \mathrm{Prob}(\{0,1\}^B)$ we may define $\Val(\mathfrak{G},\pi)$ as the game value for local data $\pi$.

For a finite $B \subset \Gamma$,  
let $\mathcal{C}_B \subset \mathrm{Prob}(\{0,1\}^B)$ be the computable polytope of local distributions that locally look like IRSs on $\Gamma$ restricted to $B$  
(as in \cite[Section~3.2]{manzoor_there_2025}, this is the image of $\mathcal{Q}_B\cap \mathcal{T}_B$ under the quotient map from the free group).

To encode the determinant conjecture locally,  
we approximate $\ln_+$ by rational polynomials:

\begin{lemma}
    For each $N \in \mathbb{N}$, there exists a sequence of rational polynomials $g_n^N$ such that:
    \begin{enumerate}
        \item $g_n^N \to \ln_+$ pointwise on $[0,N]$ as $n \to \infty$;
        \item $g_n^N(x) \geq g_{n+1}^{N'}(x)$ for all $n,N,N'$ and $x \in [0, \min(N,N')]$.
    \end{enumerate}
    $g_n^N$ can also be described computably in $(n,N)$.
\end{lemma}
\begin{proof}
    Let $f_n(t) = \max(-n, \ln(t)/t + 2^{-n})$.  
    By Stone--Weierstrass, approximate $f_n$ uniformly on $[0,N]$ within $2^{-n-1}$ by a rational polynomial $f_n^N$,  
    and set $g_n^N(t) := t f_{2n}^N(t)$. This can be done computably using Bernstein Polynomials. 
    This satisfies the stated properties.
\end{proof}

The reason this is useful is because we can write the set of traces satisfying determinant conjecture as follows:
\[T_{\mathrm{det}}(\Gamma) = \bigcap_n \{\tau \in T(\Gamma): (\tau\otimes \tr_k)((t\cdot f_n)(A^*A))\geq 0\quad \forall A\in M_k(\mathbb{Z}[\Gamma]), k\in \mathbb{N}\}.\]
So the condition of satisfying the determinant conjecture has nice upper bound. The above lemma is required to make these upper bounds ``computable''.

\begin{proof}[Proof of Theorem \ref{thm:det-npa}]
    Let $B_1 \supset K\cup\{e\}$ be finite,  
    $B_{n+1} \supset B_n$, and $\bigcup_n B_n = \Gamma$.  
    Let $\tilde{B}_n$ be all $\deg(g_n^{n^6})$ fold-products from $B_n$. This is so we can apply the polynomial to $B_n$ and end up in $\tilde{B}_n$. For $\pi \in \mathrm{Prob}(\{0,1\}^{B_n})$,  
    write $\tau_\pi(g) := \pi(\{A \subset B_n : g \in A\})$.

    Define $\mathcal{D}_n \subset \mathrm{Prob}(\{0,1\}^{\tilde{B}_n})$ to be those $\pi$ such that:
    \begin{quote}
    For every $k \leq n$ and every $A \in M_k(\mathbb{Z}[B_n])$  
    with entries having $\ell^1$-norm of coefficients $< n$, we have
    \[
        (\tau_\pi \otimes \mathrm{tr}_k) \big( g_n^{n^6}(A^*A) \big) \geq 0.
    \]
    \end{quote}
    Here $n^6$ bounds $\|A^*A\|$ since $\|A\| \leq n^3$ for such $A$.  
    This encodes a local, computable relaxation of the determinant conjecture.

    Now set
    \[
        \alpha_n := \sup \{ \Val(\mathfrak{G}, \pi) : \pi \in (R_{B_n\subset \tilde{B}_n})_*^{-1}(\mathcal{C}_{B_n} )\cap \mathcal{D}_n \}.
    \]
    The $(R_{B_n\subset \tilde{B}_n})_*^{-1}(\mathcal{C}_{B_n} )$ is because we only want to add conditions on $B_n$ but need probability measures on $\tilde{B}_n$ due to $\mathcal{D}_n$.
    \begin{itemize}
        \item Computability: This is the optimum of a finite linear program (since both $\mathcal{C}_{B_n}$ and $\mathcal{D}_n$ are described by finitely many rational inequalities that can be computed from $n$), hence computable.
        \item Monotonicity: If $m > n$, then $B_m \supset B_n$ and $g_m^{m^6} \leq g_n^{n^6}$ on $[0,n]$,  
    so $\mathcal{D}_m \subset R_{\tilde{B}_n \subset \tilde{B}_m}^{-1}(\mathcal{D}_n)$. The $\mathcal{C}$ part of the inclusion follows from \cite[Theorem 3.6, monotonicity]{manzoor_there_2025}.
    Thus $\alpha_m \leq \alpha_n$.
    \item Convergence:    The sets
    \[
        \mathcal{E}_n := (R_{\tilde{B}_n})_*^{-1}((R_{B_n\subset \tilde{B}_n})_*^{-1}(\mathcal{C}_{B_n} )\cap \mathcal{D}_n )
    \]
    are nonempty, convex, weak$^*$-compact in $\operatorname{Prob}(\mathfrak{sub}(\Gamma))$, and nested: $\mathcal{E}_{n+1} \subset \mathcal{E}_n$.  
    We claim:
    \[\bigcap_n \mathcal{E}_n = \IRS_{\mathrm{det}}(\Gamma).\]
    We already know $\bigcap_n (R_{B_n})_*^{-1}(\mathcal{C}_{B_n})=\IRS(\Gamma)$ from \cite[Theorem 3.6]{manzoor_there_2025}. All we need to check hence is that the $\mathcal{D}_n$ restrict the IRSs to satisfy the determinant conjecture. This is easy to see: if $H\in \IRS(\Gamma)$ is in the intersection of all the $(R_{\tilde{B}_n})_*^{-1}\mathcal{D}_n$ then for every $A\in M_k(\mathbb{Z}(\Gamma))$ there exists $N$ so that
    \[(\tau_H\otimes \tr_k)(g_m^{n}(A^*A))\geq 0 \quad n,m\geq N.\]
    we can find a sequence $g_{m_n}^{n}$ that will monotonically decrease and converge to $\ln_+$ on the spectrum of $A^*A$. By functional calculus we will have $(\tau_H\otimes \tr_k)(\ln_+(A^*A))\geq 0$, proving $\tau_H$ has the determinant conjecture.

    Since each of the $\mathcal{E}_n$ are compact, there is some $\pi_n\in \mathcal{E}_n$ with $\alpha_n = \Val(\mathfrak{G},\pi_n)$. Pick a $w^*$-cluster of the $\pi_n$, call it $\pi_\infty$. By passing to subsequence we can assume $\pi_n$ converge to $\pi_\infty$. Note that $\pi_\infty \in \bigcap_n\mathcal{E}_n =\IRS_{\mathrm{det}}(\Gamma)$ and that
    \[\lim_{n\to \infty} \alpha_n = \lim_{n\to \infty}\Val(\mathfrak{G},\pi_n) = \Val(\mathfrak{G},\pi_\infty) \leq \omega_{\mathrm{det}}(\mathfrak{G}).\]
    On the other hand we already know $\alpha_n\geq \omega_{\mathrm{det}}(\mathfrak{G})$ and this proves the result.
    \end{itemize}
\end{proof}

\begin{proof}[Proof of main Theorem \ref{thm:1.3}]
    For each Turing machine $\mathcal{M}$, \cite{bowen_aldous--lyons_2024} construct a game $\mathfrak{G}_\mathcal{M}$ such that:
    \begin{itemize}
        \item If $\mathcal{M}$ halts, then $\mathfrak{G}_\mathcal{M}$ has a perfect permutation strategy, hence a perfect Det-IRS strategy (see Lemma \ref{lem:2.13}.
        \item If $\mathcal{M}$ does not halt, then $\omega^*(\mathfrak{G}_\mathcal{M}) < \tfrac12$.
    \end{itemize}
    Let $\beta_n$ be the computable lower bounds for $\omega^*(\mathfrak{G}_\mathcal{M})$ coming from brute force over finite dimensional strategies.

    Suppose for contradiction that $\omega_{\mathrm{det}}(\mathfrak{G}) \leq \omega^*(\mathfrak{G})$ for all $\mathfrak{G}$.  
    Given $\mathcal{M}$, run the $\alpha_n$ and $\beta_n$ computations for $\mathfrak{G}_\mathcal{M}$ in parallel:  
    accept if some $\beta_n \geq \tfrac12$; reject if some $\alpha_n < 1$.

    If $\mathcal{M}$ halts, $\omega_{\mathrm{det}} = \omega^* = 1$,  
    so $\beta_n$ eventually exceeds $\tfrac12$ while $\alpha_n$ stays at $1$ — accept.  
    If $\mathcal{M}$ does not halt, $\omega^* < \tfrac12$ and $\omega_{\mathrm{det}} < 1$,  
    so $\beta_n$ never exceeds $\tfrac12$ and $\alpha_n$ eventually drops below $1$ — reject.  
    This would decide the halting problem, a contradiction.

    So there is a game with $\omega_{\mathrm{det}}(\mathfrak{G}) > \omega^*(\mathfrak{G})$, and by Theorem \ref{thm:3.5} we are done.
\end{proof}

\subsection{Narrowing down non--Connes embeddable algebras and IRSs}\label{sec:3.2}

Let $\mathfrak{G}$ be a non-local game with question set of size $q$ and $m$-bit answers, and set
\[
  \Gamma_{i,j} := \bigl(\mathbb{Z}_2^i \bigr)^{*j},
\]
the free product of $j$ copies of $\mathbb{Z}_2^i$. For each pair $(i,j)$, let $\mathcal{C}_{i,j}\subset T(\Gamma_{i,j})$ be a specified class of traces, and let
$\mathcal{C} = (\mathcal{C}_{i,j})_{i,j}$ denote the collection.  
Define
\[
  \omega_{\mathcal{C}}(\mathfrak{G})
  := \sup\{\operatorname{Val}(\mathfrak{G},\tau):\ \tau\in\mathcal{C}_{q,m}\},
\]
the optimal winning probability among strategies induced by traces in $\mathcal{C}_{q,m}$.  

Similarly, for IRSs let $\mathcal{H}_{i,j}\subset \IRS(\Gamma_{i,j})$ be given classes, and write
$\mathcal{H} = (\mathcal{H}_{i,j})_{i,j}$.  
Define $\omega_\mathcal{H}(\mathfrak{G})$ analogously as the optimal winning probability among IRS strategies from $\mathcal{H}_{q,m}$.

\begin{proposition}\label{prop:mip-nonemb}
\leavevmode
\begin{enumerate}
  \item[(A)] \textbf{Algebra version.}
  Suppose for the family $\mathcal{C}=(\mathcal{C}_{i,j})_{i,j}$:
  \begin{enumerate}
    \item[(i)] There exists a computably enumerable sequence $\alpha_n^{\mathcal{C}}$ decreasing monotonically to $\omega_{\mathcal{C}}(\mathfrak{G})$, and
    \item[(ii)] Each $\mathcal{C}_{i,j}$ contains all finite-dimensional traces of $\Gamma_{i,j}$.
  \end{enumerate}
  Then there exists some $(i,j)$ such that $\mathcal{C}_{i,j}$ contains a trace whose GNS von Neumann algebra is not Connes embeddable.

  \item[(B)] \textbf{IRS version.}
  Suppose for the family $\mathcal{H}=(\mathcal{H}_{i,j})_{i,j}$:
  \begin{enumerate}
    \item[(i$'$)] There exists a computably enumerable sequence $\alpha_n^{\mathcal{H}}$ decreasing monotonically to $\omega_{\mathcal{H}}(\mathfrak{G})$, and
    \item[(ii$'$)] Each $\mathcal{H}_{i,j}$ contains all finitely described IRSs of $\Gamma_{i,j}$.
  \end{enumerate}
  Then there exists some $(i,j)$ and $H\in\mathcal{H}_{i,j}$ such that $L(\Gamma_{i,j}/H)$ is not Connes embeddable (i.e.\ $H$ is not co-hyperlinear).

  \item[(C)] \textbf{Dual statement.}  
  Assume $\mathsf{MIP}^{\mathsf{co}}=\mathsf{coRE}$ (resp.\ $\mathsf{MIP}^{\mathsf{IRS}}=\mathsf{coRE}$).  
  Let $\mathcal{C}$ be a family of traces (resp. of IRSs) as above.  
  If there exists a computably enumerable sequence $\beta_n^\mathcal{C}$ increasing monotonically to $\omega_\mathcal{C}(\mathfrak{G})$, then there is some $(i,j)$ with $\mathcal{C}_{i,j}\neq T(\Gamma_{i,j})$ (resp.\ $\mathcal{H}_{i,j}\neq \IRS(\Gamma_{i,j})$).
\end{enumerate}
\end{proposition}

\begin{proof}[Proof sketch]
For (A), uncomputability of the quantum value $\omega^*$ (from $\mathsf{MIP}^*=\mathsf{RE}$\cite{ji_mipre_2022}) implies that $\omega_\mathcal{C}(\mathfrak{G})>\omega^*(\mathfrak{G})$ for some game, hence $\mathcal{C}$ must contain a non-amenable trace. This is essentially the same proof as Theorem \ref{thm:1.3}. 
For (B), the same reasoning applies using the permutation strategy version of $\mathsf{MIP}^*=\mathsf{RE}$ \cite{bowen_aldous--lyons_2024}.  
For (C), uncomputability of the commuting-operator value (resp.\ IRS value) separates it from $\omega_\mathcal{C}$ (resp.\ $\omega_\mathcal{H}$) for some game, proving that the class cannot equal the full space.
\end{proof}

The assumption $\mathsf{MIP}^{\mathsf{co}}=\mathsf{coRE}$ means that it is undecidable whether the commuting-operator value of a game is equal to $1$ or strictly less than $1/2$ (and it reduces to non-halting problem), and similarly for $\mathsf{MIP}^{\mathsf{IRS}}=\mathsf{coRE}$. This will be proved in an upcoming manuscript of Lin \cite{lin_mipco_2025}. See \cite{lin_tracial_2024} for partial progress.

Note that (C) says that the class of all von Neumann algebras or IRSs on free groups cannot have a ``nice'' approximation property that encompasses everything.

Thus Proposition~\ref{prop:mip-nonemb} gives a general recipe: any class of IRSs or traces that is rich enough to contain finite-dimensional objects and admits computable bounds must contain non–Connes embeddable examples. In this paper, we apply this principle to IRSs satisfying the determinant conjecture.

\section{Future Directions}\label{sec:4}

If one wants to give a negative solution to Problem \ref{prob:1.2}, then one approach is suggested by Proposition \ref{prop:mip-nonemb}. That is: first prove $\mathsf{MIP}^{\mathsf{IRS}}=\mathsf{coRE}$ and then prove that $\omega_{\mathrm{det}}$ has computable lower bounds. However, this is unlikely, and the way to prove this to be impossible would be to prove deciding if $\omega_{\mathrm{det}}$ is equal to $1$ or less than $1/2$ is $\mathsf{coRE}$ complete. That is to prove $\mathsf{MIP}^{\mathsf{Det-IRS}}=\mathsf{coRE}$. This would hence be further evidence of Problem \ref{prob:1.2} having a positive solution and hence of the determinant conjecture on groups being true.

It would be nice to find other natural classes to apply Proposition \ref{prop:mip-nonemb} to.

It was natural to consider the determinant conjecture in the IRS setting as soficity implies it and we know non co-sofic IRSs exist. There are other properties implied by soficity in groups, such as Gottschalk’s Surjunctivity Conjecture and Kaplansky’s Direct Finiteness Conjecture, see \cite{elek_hyperlinearity_2005}. It would be interesting to see an extension of these concepts to IRSs, but there is an obstruction in IRSs having built in analytic structure. 

A proof of Conjecture \ref{conj:2.14} will be very interesting too, even in very restricted cases.
\section*{Conflict of Interest}
On behalf of all authors, the corresponding author states that there is no conflict of interest

\bibliography{sn-bibliography}

\end{document}